\documentclass[12pt]{tran-l}

\usepackage{amssymb}

\newtheorem{theorem}{Theorem}[section]

\newtheorem{corollary}[theorem]{Corollary}

\theoremstyle{definition}

\newtheorem{example}[theorem]{Example}

\newtheorem{question}[theorem]{Question}

\theoremstyle{remark}

\numberwithin{equation}{section}

\usepackage{biblatex}
\usepackage{hyperref}
\usepackage{mathtools}

\begin{document}

\title[On the Shalit--Shamovich Spectral radius]{On the Shalit--Shamovich Spectral Radius}

\author[Maximilian Tornes]{Maximilian Tornes}
\address{Dept. of Mathematics, University of Manitoba, Winnipeg, Canada}
\curraddr{}
\email{maximilian.tornes@umanitoba.ca}
\thanks{The author is partially supported by a University of Manitoba Graduate Fellowship}

\subjclass[]{47A13, 46L07, 46L52}

\keywords{Joint spectral radius, Operator Space, multi-variable operator theory, Simultaneous similarity}

\dedicatory{}

\maketitle

\begin{abstract}
    We show that the spectral radius introduced by Shalit and Shamovich uniquely determines the underlying operator space structure: If $\mathcal{E}_1$ and $\mathcal{E}_2$ are operator space structures on $\mathbb{C}^d$ such that the corresponding spectral radii $\rho_{\mathcal{E}_1}$ and $\rho_{\mathcal{E}_2}$ agree on every $d$-tuple of operators, then the identity map is completely isometric. This answers a question of Scherer, Shalit and Shamovich.  

    Moreover, we obtain an analogous result for the minimal spectral radii, and showcase an application to the completely bounded isomorphism problem for algebras of bounded non-commutative functions on operator space unit balls.
\end{abstract}

\section{Introduction}

In \cite{ShalitShamovichSpectral}, Shalit and Shamovich associate with every operator space structure $\mathcal{E}$ on $\mathbb{C}^d$ a spectral radius function $\rho_\mathcal{E}$ on $d$-tuples of operators. Their main result, showing the ingenuity of their definition, says that a $d$-tuple $X=(X_1,\dots,X_d)\in M_n(\mathbb{C}^d)\cong M_n(\mathbb{C})^d$ of square matrices is jointly similar to an element of the open unit ball of $M_n(\mathcal{E})$ if and only if $\rho_\mathcal{E}(X)<1$. This result also holds for tuples of operators (see also \cite{SchererShalitShamovichspectral} for a proof when $\mathcal{E}$ is infinite-dimensional). When $\mathcal{E}$ is taken to be the row operator space, then $\rho_\mathcal{E}$ coincides with the well-known spectral radius considered by Bunce \cite{BunceModels}, Popescu \cite{PopescuSimilarity}, and others. 

Following up on this work, Scherer, Shalit and Shamovich \cite{SchererShalitShamovichspectral} investigated how the spectral radius $\rho_\mathcal{E}$ depends on the operator space structure $\mathcal{E}$. Among other results, they prove that the spectral radius of a commuting tuple $X$ only depends on the underlying normed space. Moreover, they show that if $\mathcal{E}_1$ and $\mathcal{E}_2$ are two selfadjoint operator space structures quantizing the same involutive Banach space structure $V$ on $\mathbb{C}^d$, then $\rho_{\mathcal{E}_1}(X)=\rho_{\mathcal{E}_2}(X)$ for all matrix tuples $X$ implies $\mathcal{E}_1=\mathcal{E}_2$. In addition, they pose the following question about the general case.

\begin{question}[{\cite[Question 1.1]{SchererShalitShamovichspectral}}]\label{question}
    Let $\mathcal{E}_1$ and $\mathcal{E}_2$ be operator space structures on $\mathbb{C}^d$ such that for every $d$-tuple of operators $X$, $\rho_{\mathcal{E}_1}(X)=\rho_{\mathcal{E}_2}(X)$. Must the identity map be completely isometric?
\end{question}

The purpose of this note is to answer Question \ref{question} affirmatively. This will be done in Corollary \ref{cor:main_result}, which is a direct consequence of Theorem \ref{thm:technical_result}, in which we explicitly construct for a given operator $d$-tuple $X$, an operator $d$-tuple $T$ such that $\rho_{\mathcal{E}}(T)=\|X\|_\mathcal{E}$ for every operator space structure $\mathcal{E}$ on $\mathbb{C}^d$. We will also prove that Question \ref{question} still has a positive answer if one replaces $\rho_{\mathcal{E}_1}$ and $\rho_{\mathcal{E}_2}$ by the minimal spectral radii as defined by Shalit and Shamovich. All necessary definitions will be recalled in Section \ref{sec:Background}. As an application, in Section \ref{sec:application} we show that if $A(\mathbb{B}_{\mathcal{E}_1})=A(\mathbb{B}_{\mathcal{E}_2})$ as sets of functions, with the identity being completely bounded, then the identity is necessarily completely isometric (which in turn is equivalent to the identity map $\mathcal{E}_1\to\mathcal{E}_2$ being completely isometric). Here, $A(\mathbb{B}_{\mathcal{E}_i})$ denotes the algebra of bounded analytic non-commutative functions on the open unit ball $\mathbb{B}_{\mathcal{E}_i}$ of $\mathcal{E}_i$ that extend to uniformly continuous functions on the closed unit ball.

\begin{example}\label{ex:row_column}
    Denote by $\mathcal{E}_\mathrm{row}$ and $\mathcal{E}_\mathrm{col}$ the row and column operator space structures on $\mathbb{C}^d$, respectively. In this case, one has $\rho_{\mathcal{E}_\mathrm{row}}(X)=\rho_{\mathcal{E}_\mathrm{col}}(X)$ for all matrix tuples $X$. There is, however, a tuple of operators on which the two spectral radii do not agree. Moreover, the identity map on $\mathbb{C}^d$ does not induce a completely isometric map $\mathcal{E}_\mathrm{row}\to \mathcal{E}_\mathrm{col}$ for $d\geq 2$ (See \cite[Example 5.2]{SchererShalitShamovichspectral} for details).

    Therefore, going back to Question \ref{question}, the fact that $\rho_{\mathcal{E}_1}(X)=\rho_{\mathcal{E}_2}(X)$ for all operator tuples $X$, not just all matrix tuples $X$, must form an integral part of any proof that yields an affirmative answer to Question \ref{question}. And indeed, the tuple $T$ we construct in Theorem \ref{thm:technical_result} acts on an infinite dimensional Hilbert space. 
\end{example}

\section{Background}\label{sec:Background}

Let us recall how Shalit and Shamovich define their spectral radius functions in \cite{ShalitShamovichSpectral}. We refer to \cite{PisierBook,PaulsenBook} for background on operator spaces and completely bounded maps. Let $\mathcal{E}$ be an operator space and denote by $\iota$ the completely isometric embedding $\iota:\mathcal{E}\to C^*_\mathrm{max}(\mathcal{E})$. We also use $\iota$ to denote ampliations $B(K)\otimes_\mathrm{min} \mathcal{E}\to B(K)\otimes_\mathrm{min}C^*_\mathrm{max}(\mathcal{E})$ of this inclusion. The spectral radius of an element $X\in B(K)\otimes_\mathrm{min} \mathcal{E}$ is then defined as its spectral radius when viewed as an element of the $C^*$-algebra $B(K)\otimes_\mathrm{min} C^*_\mathrm{max}(\mathcal{E})$:
\begin{align}\label{eq:spectral_rad_def}
    \rho_\mathcal{E}(X)=\rho(\iota(X))=\lim_{n\to\infty}\|\iota(X)^n\|^{1/n}.
\end{align}
When $\mathcal{E}$ is finite dimensional, viewed as $\mathbb{C}^d$ equipped with a family of matrix norms (all of which we denote by $\|\cdot\|_\mathcal{E}$), one can obtain a more concrete, coordinate dependent version of this spectral radius. For $j=1,\dots,d$, denote by $Q_j$ the image of the standard basis vector $e_j\in\mathbb{C}^d$ under a completely isometric embedding $\mathcal{E}=\mathbb{C}^d\to B(H)$. This identification also determines a norm on $d$-tuples of operators $X=(X_1,\dots,X_d)$, given by $\|X\|_\mathcal{E}:=\|\sum_{j=1}^d X_j\otimes Q_j\|$. Next, one defines for $X=(X_1,\dots,X_d)\in B(K)^d$ the spectral radius as
\begin{align}\label{eq:equality_of_radii}
    \rho_\mathcal{E}(X)=\rho_Q(X):=\rho\left(\sum_{j=1}^d X_j\otimes \iota(Q_j)\right).
\end{align}
Recall that $\rho$ denotes the ordinary spectral radius of the $C^*$-algebra $B(K)\otimes_\mathrm{min} C^*_\mathrm{max}(\mathcal{E})$. Shalit and Shamovich show that $\rho_Q$, which is called the $Q$-\emph{spectral radius}, is given by
\begin{align}\label{eq:Q_spectral_radius}
    \rho_Q(X)=\lim_{n\to\infty}\left\|\sum_{|\omega|=n} X^\omega\otimes Q_{\omega_1}\otimes_h\cdots\otimes_h Q_{\omega_n}\right\|^{1/n}.
\end{align}
In \eqref{eq:Q_spectral_radius}, the sum is taken over all words $\omega=\omega_1\omega_2\cdots\omega_n$ of length $n$ in the letters $\omega_j\in\{1,\dots,d\}$, with $X^\omega=X_{\omega_1}X_{\omega_2}\cdots X_{\omega_n}$. The norm in \eqref{eq:Q_spectral_radius} is the norm in $B(K)\otimes_\mathrm{min}(\mathcal{E}\otimes_h \cdots\otimes_h \mathcal{E})$. Here, $\mathcal{E}\otimes_h\cdots\otimes_h \mathcal{E}$ is the $n$-fold Haagerup tensor product of $\mathcal{E}$ with itself (We refer the reader to \cite[Chapter 5]{PisierBook} for background on the Haagerup tensor product). When $\mathcal{E}$ is the row operator space on $\mathbb{C}^d$, then \eqref{eq:Q_spectral_radius} agrees with the spectral radius considered by Bunce, Popescu, and others (see \cite[Example 2.10]{ShalitShamovichSpectral}). We will use $\mathbb{B}_\mathcal{E}$ to denote the unit ball of $\mathcal{E}$, i.e., the set of all $d$-tuples $X$ of square matrices of any size, for which $\|X\|_\mathcal{E}<1$. Similarly, $\overline{\mathbb{B}_\mathcal{E}}$ denotes the set of all matrix $d$-tuples $X$ for which $\|X\|_\mathcal{E}\leq 1$. The main strength of the spectral radius \eqref{eq:equality_of_radii} is that it detects when a matrix tuple is jointly similar to an element of the unit ball of $\mathcal{E}$.
 \begin{theorem}[{Shalit--Shamovich \cite{ShalitShamovichSpectral}}]\label{thm:main_SS_result}
    A $d$-tuple of matrices $X=(X_1,\dots,X_d)\in M_n(\mathbb{C}^d)$ satisfies $\rho_\mathcal{E}(X)<1$ if and only if there is an invertible matrix $S\in \mathrm{GL}_n$ such that $SXS^{-1}=(SX_1S^{-1},\dots,SX_dS^{-1})\in \mathbb{B}_\mathcal{E}$.
\end{theorem}
Theorem \ref{thm:main_SS_result} also holds for operator tuples in place of matrix tuples \cite[Theorem 2.7]{ShalitShamovichSpectral}, and even when $\mathcal{E}$ is infinite dimensional \cite[Theorem 2.1]{SchererShalitShamovichspectral}.

In \cite{ShalitShamovichSpectral}, Shalit and Shamovich also introduce a different notion of spectral radius: Given operators $Q_1,\dots,Q_d\in B(H)$ as above, the $\emph{minimal $Q$-spectral radius}$ of a tuple $X=(X_1,\dots,X_d)\in B(K)^d$ is
\begin{align}\label{eq:minimal_Q_spectral_radius}
    \rho_Q^\mathrm{min}(X)=\lim_{n\to\infty} \left\|\sum_{|\omega|=n} X^\omega\otimes Q_{\omega_1}\otimes\cdots\otimes Q_{\omega_n}\right\|^{1/n}. 
\end{align}
Here, every tensor product is the spatial (minimal) tensor product of operators. In general, a tensor product without a subscript will always stand for the spatial (minimal) tensor product. When we take $\mathcal{E}=\mathrm{min}(\ell^\infty_d)$, i.e., the minimal operator space structure over $\ell^\infty_d=(\mathbb{C}^d,\|\cdot\|_\infty)$, then the spectral radius in \eqref{eq:minimal_Q_spectral_radius} agrees with spectral radius studied by Rota and Strang \cite{RotaStrang}. In general, the $Q$-spectral radius and the minimal $Q$-spectral radius differ (See \cite[Example 2.11]{ShalitShamovichSpectral}). Going forward, whenever $\mathcal{E}$ is an operator space structure on $\mathbb{C}^d$ and $Q_1,\dots,Q_d$ are as above, we will write $\rho_\mathcal{E}^\mathrm{min}$ instead of $\rho_Q^\mathrm{min}$.
\section{Main Result}\label{sec:main_result}

As described in the introduction, for every operator $d$-tuple $X$, we will construct an operator $d$-tuple $T$ such that $\rho_\mathcal{E}(T)=\|X\|_\mathcal{E}$ for all operator space structures $\mathcal{E}$ on  $\mathbb{C}^d$. If $\{Q_1,\dots,Q_d\}$ are operators that form a basis of a given $\mathcal{E}$ when viewed as a concrete subspace of some $B(H)$, and $X=(X_1,\dots,X_d)$ is an operator $d$-tuple, we will write $Q(X)$ in place of $\sum_{j=1}^d X_j\otimes Q_j$.

\begin{theorem}\label{thm:technical_result}
    Let $X=(X_1,\dots,X_d)\in B(H)^d$ be a tuple of operators on some Hilbert space $H$. Then there exists an infinite dimensional Hilbert space $K$ and a tuple of operators $T=(T_1,\dots,T_d)\in B(K)^d$ with the following property: For every operator space structure $\mathcal{E}$ on $\mathbb{C}^d$, every basis $\{Q_1,\dots,Q_d\}$ of $\mathcal{E}$, and all $n\in\mathbb{N}$, 
    \begin{align*}
        \left\|\sum_{|\omega|=n} T^\omega\otimes Q_{\omega_1}\otimes_h\cdots\otimes_h Q_{\omega_n}\right\|= \|Q(X)\|^n,
    \end{align*}
    and 
     \begin{align*}
        \left\|\sum_{|\omega|=n} T^\omega\otimes Q_{\omega_1}\otimes\cdots\otimes Q_{\omega_n}\right\|= \|Q(X)\|^n.
    \end{align*}
\end{theorem}

\begin{proof}
    Let $k\in \mathbb{N}$ and let $U_k$ be the unitary operator on $H^{\otimes k}=H\otimes \cdots \otimes H$, the $k$-fold tensor product of $H$ with itself, which implements the permutation $a_1\otimes a_2\otimes \cdots\otimes  a_k\mapsto a_k\otimes a_1 \otimes a_2\otimes \cdots\otimes a_{k-1}$. Next, for $j=1,\dots,d$, define
    \begin{align*}
        T_j^{(k)}=(X_j\otimes I\otimes \cdots \otimes I)U_k \in B(H^{\otimes k}),
    \end{align*}
    and $T^{(k)}=(T_1^{(k)},\dots, T_d^{(k)})$. Then, for every word $\omega=\omega_1\omega_2\cdots\omega_k$ of length $k$ in the letters $\omega_j\in\{1,\dots,d\}$, we have
    \begin{align*}
        (T^{(k)})^\omega=T_{\omega_1}^{(k)}T_{\omega_2}^{(k)}\cdots T_{\omega_k}^{(k)}=X_{\omega_1}\otimes X_{\omega_2}\otimes\cdots \otimes X_{\omega_k}.
    \end{align*}
    This can be readily checked by applying $(T^{(k)})^\omega$ to elementary tensors. For $j=1,\dots,d$, define the following operator:
    \begin{align*}
        T_j=\bigoplus_{k=1}^\infty T_j^{(k)}\in B\left(\bigoplus_{k=1}^\infty H^{\otimes k}\right).
    \end{align*}
    Moreover, define $T=(T_1,\dots,T_d)$.
    Let $\mathcal{E}$ be an operator space structure on $\mathbb{C}^d$ and let $\{Q_1,\dots,Q_d\}$ be a basis of $\mathcal{E}$. A calculation shows that $\|Q(T^{(k)})\|=\|Q(X)\|$ for all $k\in\mathbb{N}$. Hence, $\|Q(T)\|=\|Q(X)\|$. Fix $n\in\mathbb{N}$. The equality $\|Q(T)\|=\|Q(X)\|$, together with the definition of the Haagerup tensor product \cite[Equation (5.6)]{PisierBook} yields
    \begin{align}\label{eq:mainthm1}
        \left\|\sum_{|\omega|=n} T^\omega\otimes Q_{\omega_1}\otimes_h\cdots\otimes_h Q_{\omega_n}\right\|\leq \|Q(X)\|^n.
    \end{align}
    Next, we use the fact that the Haagerup tensor norm dominates the minimal tensor norm to get
    \begin{align}\label{eq:mainthm2}
         &\left\|\sum_{|\omega|=n} T^\omega\otimes Q_{\omega_1}\otimes_h\cdots\otimes_h Q_{\omega_n}\right\|\\
         \geq&  \left\|\sum_{|\omega|=n} T^\omega\otimes Q_{\omega_1}\otimes\cdots\otimes Q_{\omega_n}\right\|.\nonumber
    \end{align} 
    Finally, using the definition of $T$,
    \begin{align}\label{eq:mainthm3}
         &\left\|\sum_{|\omega|=n} T^\omega\otimes Q_{\omega_1}\otimes\cdots\otimes Q_{\omega_n}\right\|\\
         \geq &  \left\|\sum_{|\omega|=n} (T^{(n)})^\omega\otimes Q_{\omega_1}\otimes\cdots\otimes Q_{\omega_n}\right\|\nonumber\\
         =&  \left\|\sum_{|\omega|=n}X_{\omega_1}\otimes\cdots\otimes X_{\omega_n} \otimes Q_{\omega_1}\otimes\cdots\otimes Q_{\omega_n}\right\|\nonumber\\
         =& \left\|\sum_{|\omega|=n}X_{\omega_1}\otimes Q_{\omega_1} \otimes \cdots\otimes X_{\omega_n}\otimes Q_{\omega_n}\right\|\nonumber\\
         =& \left\|\left(\sum_{j=1}^d X_j\otimes Q_j\right)^{\otimes n}\right\|=\|Q(X)\|^n.\nonumber
    \end{align}
    The claim now follows by combining (\ref{eq:mainthm1}), (\ref{eq:mainthm2}) and (\ref{eq:mainthm3}).
\end{proof}

\begin{corollary}
     Let $X=(X_1,\dots,X_d)\in B(H)^d$ be a tuple of operators on some Hilbert space $H$. Then there exists an infinite dimensional Hilbert space $K$ and a tuple of operators $T=(T_1,\dots,T_d)\in B(K)^d$ such that $\rho_\mathcal{E}(T)=\rho_\mathcal{E}^\mathrm{min}(T)=\|X\|_\mathcal{E}$ for all operator space structures $\mathcal{E}$ on $\mathbb{C}^d$.
\end{corollary}

\begin{proof}
    Let $X=(X_1,\dots,X_d)\in B(H)^d$ and let $T=(T_1,\dots,T_d)$ be as in Theorem \ref{thm:technical_result}. 
    Next, for a given operator space structure $\mathcal{E}$ on $\mathbb{C}^d$, let $J$ be a Hilbert space such that $\mathcal{E}\subset B(J)$ completely isometrically, and for $j=1,\dots,d$ let $Q_j$ be the image of the standard basis vector $e_j\in\mathbb{C}^d$ under this embedding. Then, $\|X\|_{\mathcal{E}}=\|Q(X)\|=\|\sum_{j=1}^d X_j\otimes Q_j\|$, and  
    \begin{align*}
        \rho_{\mathcal{E}}(T)=\lim_{n\to \infty}\left\|\sum_{|\omega|=n} T^\omega\otimes Q_{\omega_1}\otimes_h \cdots\otimes_h Q_{\omega_n}\right\|^{1/n},
    \end{align*}
    according to \eqref{eq:Q_spectral_radius} and \eqref{eq:equality_of_radii}. Using Theorem \ref{thm:technical_result}, we conclude that $\rho_\mathcal{E}(T)=\|X\|_\mathcal{E}$.

    The statement about the minimal spectral radius follows similarly from its definition (see \eqref{eq:minimal_Q_spectral_radius}) and Theorem \ref{thm:technical_result}.
\end{proof}

As an immediate consequence, we obtain the promised answer to Question \ref{question}.

\begin{corollary}\label{cor:main_result}
    Let $\mathcal{E}_1$ and $\mathcal{E}_2$ be operator space structures on $\mathbb{C}^d$. If $\rho_{\mathcal{E}_1}(X)=\rho_{\mathcal{E}_2}(X)$ for every $d$-tuple of operators $X$, then the identity map $\mathcal{E}_1\to \mathcal{E}_2$ is completely isometric.

    The statement remains true if we replace the spectral radii $\rho_{\mathcal{E}_1}$ and $\rho_{\mathcal{E}_2}$ by the minimal spectral radii $\rho_{\mathcal{E}_1}^\mathrm{min}$ and $\rho_{\mathcal{E}_2}^\mathrm{min}$.
\end{corollary}

\section{An Application}\label{sec:application}

    In \cite[Section 8]{SchererShalitShamovichspectral}, Scherer, Shalit and Shamovich investigate non-commutative (NC) function theoretic consequences of $\rho_{\mathcal{E}_1}=\rho_{\mathcal{E}_2}$. NC function theory can be thought of as a generalization of several variable complex analysis to analytic functions in several non-commuting variables (see \cite{kaliuzhnyi2014foundations} for a detailed treatment). Polynomials in several non-commuting variables are basic examples of NC functions. For a given operator space structure $\mathcal{E}$ on $\mathbb{C}^d$, denote by $\mathcal{O}(\mathbb{B}_\mathcal{E})$ the algebra of all NC functions $f$ defined on $\mathbb{B}_\mathcal{E}$, such that $f$ is uniformly bounded on $r\overline{\mathbb{B}_\mathcal{E}}$ for every $0<r<1$, equipped with the topology of uniform convergence on closed sub-balls. Scherer, Shalit and Shamovich show that $\rho_{\mathcal{E}_1}(X)=\rho_{\mathcal{E}_2}(X)$ for all operator $d$-tuples $X$ implies $\mathcal{O}(\mathbb{B}_{\mathcal{E}_1})=\mathcal{O}(\mathbb{B}_{\mathcal{E}_2})$ as sets of functions, with the identity map being a complete homeomorphism.  

    Moreover, they ask whether it is possible to obtain the equality $A(\mathbb{B}_{\mathcal{E}_1})=A(\mathbb{B}_{\mathcal{E}_2})$ from equality of the corresponding spectral radii (see \cite[Remark 8.5]{SchererShalitShamovichspectral}). Here, for a given operator space structure $\mathcal{E}$ on $\mathbb{C}^d$, $A(\mathbb{B}_{\mathcal{E}})$ denotes the closure of the NC polynomials in $H^\infty(\mathbb{B}_{\mathcal{E}})$, the Banach algebra of bounded NC functions on $\mathbb{B}_{\mathcal{E}}$ equipped with the supremum norm $\|f\|=\sup_{X\in\mathbb{B}_{\mathcal{E}}}\|f(X)\|$. The algebra $A(\mathbb{B}_\mathcal{E})$ can be viewed as an NC analogue of the disk algebra $A(\mathbb{D})$. By \cite{SampatShalitClassification}, $A(\mathbb{B}_{\mathcal{E}})$ equals the set of bounded NC functions on $\mathbb{B}_{\mathcal{E}}$ that extend to uniformly continuous functions on $\overline{\mathbb{B}_{\mathcal{E}}}$. Moreover, $A(\mathbb{B}_{\mathcal{E}})$ is completely isometrically isomorphic to $\mathrm{OA}_u(\mathcal{E}^*)$, the universal unital operator algebra generated by the operator space dual $\mathcal{E}^*$ of $\mathcal{E}$ (see \cite[Section 7.1]{SampatShalitClassification}). Scherer, Shalit and Shamovich prove that $\rho_{\mathcal{E}_1}(X)=\rho_{\mathcal{E}_2}(X)$ for all matrix $d$-tuples $X$ is not sufficient to imply $A(\mathbb{B}_{\mathcal{E}_1})=A(\mathbb{B}_{\mathcal{E}_2})$. This is done by considering the row and column operator spaces that we encountered in Example \ref{ex:row_column}. Corollary \ref{cor:main_result} yields an answer to the aforementioned question from \cite[Remark 8.5]{SchererShalitShamovichspectral}:  It is true that if $\rho_{\mathcal{E}_1}(X)=\rho_{\mathcal{E}_2}(X)$ for every $d$-tuple of operators $X$, then $A(\mathbb{B}_{\mathcal{E}_1})=A(\mathbb{B}_{\mathcal{E}_2})$.

    The problem of classifying algebras of bounded analytic functions, commutative as well as non-commutative, has been of interest to several authors in recent decades (see \cite{SalomonShalitSurvey}, for example, and the references therein for the isomorphism problem for complete Pick algebras). In the non-commutative setting, if $\mathfrak{V}$ denotes an NC subvariety of some operator space unit ball $\mathbb{B}_\mathcal{E}$, one of the goals is to classify the algebras $A(\mathfrak{V})$ and $H^\infty(\mathfrak{V})$ up to completely bounded or completely isometric isomorphism. The completely bounded and completely isometric isomorphism problem for subvarieties of the row ball has been investigated by Salomon, Shalit and Shamovich \cite{SSS1,SSS2}. Classification up to completely isometric isomorphism for subvarieties of general operator space unit balls $\mathbb{B}_\mathcal{E}$ has been taken up by Sampat and Shalit \cite{SampatShalitClassification,SampatShalitweak}. What remains is the completely bounded isomorphism problem for subvarieties of general operator space unit balls $\mathbb{B}_\mathcal{E}$. The results of Section \ref{sec:main_result} have the following application to this problem.

\begin{corollary}
    If $\mathcal{E}_1$ and $\mathcal{E}_2$ are operator space structures on $\mathbb{C}^d$, then the following are equivalent:
    \begin{enumerate}
        \item $A(\mathbb{B}_{\mathcal{E}_1})=A(\mathbb{B}_{\mathcal{E}_2})$ as sets of functions, with the identity being completely bounded;
        \item $A(\mathbb{B}_{\mathcal{E}_1})=A(\mathbb{B}_{\mathcal{E}_2})$ as sets of functions, with the identity being completely isometric;
        \item the identity map $\mathcal{E}_1\to\mathcal{E}_2$ is completely isometric.
    \end{enumerate}
\end{corollary}

\begin{proof}
    The implications (3)$\implies$(2)$\implies$(1) are clear. 

    Assume that (1) holds, we will follow the approach taken in \cite[Lemma 2.5]{SchererShalitShamovichspectral}. More precisely, we will show, that for all $d$-tuples of operators $X=(X_1,\dots,X_d)\in B(H)^d$ , $\rho_{\mathcal{E}_2}(X)<1$ implies $\rho_{\mathcal{E}_1}(X)\leq 1$. Since the spectral radii are homogeneous, this will then yield $\rho_{\mathcal{E}_1}(X)\leq \rho_{\mathcal{E}_2}(X)$. Denote by $Z_1,\dots,Z_d$ the coordinate functions in $A(\mathbb{B}_{\mathcal{E}_1})=A(\mathbb{B}_{\mathcal{E}_2})$. If $\rho_{\mathcal{E}_2}(X)<1$, then by \cite[Lemma 2.4]{SchererShalitShamovichspectral} the assignment $Z_j\mapsto X_j$ gives rise to a completely bounded unital homomorphism $\varphi:A(\mathbb{B}_{\mathcal{E}_2})\to B(H)$. Using (1), we get a completely bounded unital homomorphism 
    \begin{align*}
        A(\mathbb{B}_{\mathcal{E}_1})\overset{\mathrm{id}}{\longrightarrow} A(\mathbb{B}_{\mathcal{E}_2})\overset{\varphi}{\longrightarrow} B(H),
    \end{align*}
    mapping $Z_j$ to $X_j$. Thus, $\rho_{\mathcal{E}_1}(X)\leq 1$ by \cite[Lemma 2.2]{SchererShalitShamovichspectral} as desired.

    The inequality $\rho_{\mathcal{E}_2}(X)\leq \rho_{\mathcal{E}_1}(X)$ follows from a similar argument. Since the spectral radii are now shown to be equal on every operator $d$-tuple $X$, the identity $\mathcal{E}_1\to\mathcal{E}_2$ must be completely isometric by Corollary \ref{cor:main_result}. 
\end{proof}

\subsubsection*{Acknowledgments}
We would like to thank Robert T.W. Martin, Jeet Sampat and Orr Moshe Shalit for several helpful discussions and feedback on this article.

\subsubsection*{AI Disclosure}
ChatGPT 5.6 was used for literature searches, proofreading, and to help explore some of our ideas for answering Question \ref{question}. The article and its proofs are entirely written by the author.
\printbibliography

\end{document}